\documentclass[11pt]{article}

\usepackage{epic,latexsym,amssymb}
\usepackage{color}
\usepackage{tikz}
\usepackage{amsmath,amsthm}

\newtheorem{thm}{Theorem}
\newtheorem{lem}{Lemma}

\newtheorem{conj}{Conjecture}

\allowdisplaybreaks

\newcommand{\Z}{{\Z B}}

\newcommand{\Zekhaya}[1]{{\color{red} #1}}

\let\oldenumerate\enumerate
\renewcommand{\enumerate}{
  \oldenumerate
  \setlength{\itemsep}{0pt}
  \setlength{\parskip}{0pt}
  \setlength{\parsep}{0pt}
}

\def\vertex(#1){\put(#1){\circle*{2}}}
\def\vertexo(#1){\put(#1){\circle{2}}}
\def\vert(#1){\put(#1){\circle*{1.5}}}
\def\verto(#1){\put(#1){\circle{1.5}}}
\def\lab(#1)#2{\put(#1){\makebox(0,0)[c]{#2}}}

\begin{document}

\title{ The number of $1$-nearly independent edge subsets}

\author{Eric O. D. Andriantiana \thanks{We would like to thank the South African Department of Higher Education and Training’s Future Professors Programme Phase 01 for their financial support for E. O. D Andriantiana.} \\
	Department of Mathematics (Pure and Applied) \\
	Rhodes University \\
	Makhanda, 6140 South Africa\\
	\small \tt Email: e.andriantiana@ru.ac.za \\\\Zekhaya B. Shozi \thanks{Research supported by the National Graduate Academy for Mathematical and Statistical Sciences (NGA-MaSS). Grant number: PA22023.}\\
	School of Mathematics, Statistics and Computer Science\\
	University of KwaZulu-Natal\\
	Durban, 4000 South Africa\\
\small \tt Email: zekhaya@aims.ac.za
}

\date{}
\maketitle

\begin{abstract}

Let $G=(V(G),E(G))$ be a graph with set of vertices $V(G)$ and set of edges $E(G)$. A subset $S$ of $E(G)$ is called a $k$-nearly independent edge subsets if there are exactly $k$ pairs of elements of $S$ that share a common end. $Z_k(G)$ is the number of such subsets. 
This paper studies $Z_1$. Various properties of $Z_1$ are discussed. We characterise the two $n$-vertex trees with smallest $Z_1$, as well as the one with largest value. A conjecture on the $n$-vertex tree with second-largest $Z_1$ is proposed.

 \end{abstract}

\newpage
	
\section{Introduction}

A \emph{simple} and \emph{undirected graph} $G$ is an ordered pair of sets $G=(V(G),E(G))$, where $V(G)$ is a 
set of objects called \emph{vertices}, and $E(G)$ is a (possibly empty) set of unordered pairs of elements of $V(G)$ called \emph{edges}. The \emph{order} and \emph{size} of $G$ is $|V(G)|$ and $|E(G)|$, respectively. For simplicity, we write $|G|$ instead of $|V(G)|$. For graph theory notation and terminology, we generally follow~\cite{henning2013total}.

An \emph{independent edge subset} of a graph $G=(V(G),E(G))$ is a subset $I$ of $V(G)$ with the property that if $e_1$ and $e_2$ are two edges in $I$, then $e_1$ and $e_2$ are not adjacent in $G$; that is, $e_1$ and $e_2$ do not share a common end in $G$. Similarly, as already defined in \cite{andriantiana2023number}, $\sigma_1(G)$ counts the number of independent subsets $S$ of $V(G)$ such that the subgraph induced by $S$ in $G$ contains only one edge.  The number of independent edge subsets of a graph has been extensively studied in the literature. See the survey  \cite{wagner2010maxima}, where it is called the \emph{Hosoya index}. The name Hosoya comes from the Japanese chemist, Haruo Hosoya \cite{hosoya1971topological} who was the first person to introduce this index in 1971. He showed that there is a correlation between the boiling points of paraffins (saturated hydrocarbons) and the Hosoya index. 

In a series of papers \cite{gao1988topological, gutman2004concealed, gutman1976topological, hosoya1972graphical, hosoya1976topological, hosoya1975topological}, the application of the Hosoya index was revealed in showing the structure-dependence of the total $\pi$-electron energy of chemical molecules. This boosted the interest of many mathematicians to study the number of independent edge subsets. Established results include classes of graphs that contain elements that are not molecular graphs. In \cite{andriantiana2013energy}, the family of trees with a given degree sequence is studied, and the element which has the minimum number of independent edge subsets is fully characterised.  The result implies as corollaries characterisations of trees with smallest
number of independent edge subsets in various other classes like trees with fixed order, or with fixed order and given maximum degree.

This paper intends to propose a generalisation of the number of independent edge subsets. For an integer $k\ge 1$, we define a \emph{$k$-nearly independent edge subset} of a graph $G$ with vertex set $V(G)$ and edge set $E(G)$ as a subset $I$ of $E(G)$ that contains exactly $k$ pairs of adjacent edges. Denote by $Z_k(G)$ the number of $k$-nearly independent edge subsets of $G$. $Z_0(G)$ is the number of independent edge subsets of $G$. The main focus of this paper is to study $Z_1$. For the classes of graphs we investigated, the behavior of $Z_1$ seems to have a lot in common with that of $Z_0$. Among all trees of order $n$, the star $K_{1,n-1}$ minimises both $Z_0$ and $Z_1$, while the path $P_n$ maximises both $Z_0$ and $Z_1$. In \cite{andriantiana2023number} similar comparison made between $\sigma_1$ and the usual number of independent vertex subsets shows considerable difference.

The rest of the paper is structured as follows. Section \ref{Sec:Prel} is a preliminary, where we present basic useful facts about $Z_1$. There, we discuss the effect of adding or removing an edge, we provide recursive formulas for $Z_1$ as well as an explicit formula for $Z_1$ of paths. These are used in Section \ref{Sec:Min} to characterise the two trees with order $n$ and smallest $Z_1.$ In Section \ref{Sec:Max}, we proved that the path $P_n$ is the forest of order $n$ that has the largest $Z_1$. A conjecture on the forest with second-largest $z_1$ is also provided there. 


\section{Preliminary}
\label{Sec:Prel}
This section is made of a few technical tools that will be needed in other sections. Since we are introducing $Z_1$, we also include some properties that we do not use much, but we expect to be useful for further studies of $Z_1$. 

Let $G$ be a graph with vertex set $V(G)$, edge set $E(G)$, order $n = |V(G)|$ and size $m = |E(G)|$. We denote the degree of a vertex $v$ in $G$ by $\deg_G(v)$. For a subset $S$ of vertices of a graph $G$, we denote by $G - S$ the graph obtained from $G$ by deleting the vertices in $S$ and all edges incident to them. If $S = \{v\}$, then we simply write $G - v$ rather than $G - \{v\}$.

For positive integers $r$ and $s$, we denote by $K_{r,s}$ the complete bipartite graph with partite sets $X$ and $Y$ such that $|X|=r$ and $|Y|=s$. A complete bipartite graph $K_{1,n-1}$ is also called a \emph{star}. 
We use the typical notations $P_n, C_n, $ and $K_n$ 
for the path of order $n$, the cycle of order $n$ and the complete graph of order $n$, 
respectively.

\subsection{Effect of an edge or vertex removal or addition}
\label{Sub:REmE}
Suppose that $u$ and $v$ are vertices not adjacent in a graph $G$. If $H=G+uv$ is the graph obtained from $G$ by adding the edge $uv$, then $Z_1(H)\geq Z_1(G)$. Adding the edge $uv$ does not affect the adjacency of the already existing edges. The inequality is strict if and only if $G-u-v$ contains a $P_3$ or at least one of $u$ and $v$ is not isolated. These are the only situations where the newly added edge is contained in at least one new set of $1$-nearly independent edges.   For small graphs like $2K_1$ (two vertices with no edge), it is possible that the newly added edge is not contained in a $1$-nearly independent edge subset. In this case $Z_1$ will not increase. 
It follows from this that for any graph $G$ with $n$ vertices, we have $$Z_1(G)\leq Z_1(K_n)$$ with equality only if the two compared graphs are the same or $n<3$. The edgeless graph has the smallest $Z_1$ which is $0$.

An isolated vertex does not affect the value of $Z_1$. If $\deg_G(v)=0$, then $Z_1(G)=Z_1(G-v).$ If $\deg_G(v)\geq 2$, then $Z_1(G)>Z_1(G-v)$. In this case, at least one possibility of a $P_3$ subtree of $G$ is lost and hence at least one $1$-nearly independent edge subsets. We still have $Z_1(G)>Z_1(G-v)$ if we remove a vertex $v$ of degree $1$ that is attached to a vertex of degree at least $2$, as we then lose at least one $P_3$ containing $v$. If $v$ is contained in $P_2$ component of $G$, then $Z_1(G)\geq Z_1(G-v)$, with strict inequality if $G-N_G[v]$ contains a $P_3$. 
\subsection{Relation between $\sigma_1$ and $Z_1$}
The line graph $L(G)$ of $G$ is the graph with set of vertices $E(G)$, and such that two different elements  $e$ and $e'$ of $E(G)$
are adjacent in $L(G)$ if they have a common end in $G$. We state, without proof, the following  straightforward lemma.
\begin{lem}
\label{Lem:GtoLG}
For any graph $G$, we have 
$Z_1(G)=\sigma_1(L(G))$.
\end{lem}

\subsection{Explicit formulas for $Z_1$ of some graphs}

It is convenient to set $Z_1(P_t)=0$ whenever $t\leq 2$ and $Z_0(P_t)=1$ whenever $t\leq 1$.
Let $\alpha = \frac{1+\sqrt{5}}{2}$ and $\beta = \frac{1-\sqrt{5}}{2}$. Then, we have 
    $\alpha + \beta =1$,
     $\alpha - \beta = \sqrt{5}$,
     $\alpha\cdot \beta = -1$.
The following formulas are well known, see for example \cite{andriantiana2010number} and \cite{wagner2010maxima}. 

\begin{thm}[cf. \cite{andriantiana2010number}]
\label{thm:sigma-of-a-path-by-eric}
    For $n \in \mathbb{N}$, 
    we have
    \begin{align}
    \label{eq:sigma-of-a-path-by-eric}
        Z_0(P_n) = \frac{1}{\sqrt{5}} \left( \alpha^{n+1} - \beta^{n+1} \right).
    \end{align}
\end{thm}

\begin{thm}[cf. \cite{wagner2010maxima}]
    \label{thm:sigma-union-is-product-of-sigmas}
    If $G_1, G_2, \ldots, G_r$ are the connected components of a graph $G$, then
    \begin{align}
        \label{eq:sigma-union-is-product-of-sigmas}
        Z_0(G) = Z_0\left(\bigcup\limits_{i=1}^r G_i \right) = \displaystyle \prod\limits_{i=1}^r Z_0(G_i).
    \end{align}
\end{thm}


The following results have been recently established \cite{andriantiana2023number}.
\begin{align*}
    \sigma_1(P_n) = \frac{1}{5}\left[(n-1)\left(\alpha^n + \beta^n \right) + \frac{2}{\sqrt{5}} \left( \alpha^{n-1} -\beta^{n-1} \right) \right].
\end{align*}
and
$$
    \sigma_1(C_n)  = \frac{n}{\sqrt{5}}\left( \alpha^{n-2} - \beta^{n-2} \right).
$$

Thanks to Lemma \ref{Lem:GtoLG}, we also have
\begin{align*}
    Z_1(P_n) &= 
    \sigma_1(L(P_n))=\sigma_1(P_{n-1})
   = \frac{1}{5}\left[(n-2)\left(\alpha^{n-1} + \beta^{n-1} \right) + \frac{2}{\sqrt{5}} \left( \alpha^{n-2} -\beta^{n-2} \right) \right]
\end{align*}
and
\begin{align*}
    Z_1(C_n) &= 
    \sigma_1(L(C_n))=\sigma_1(C_{n})= \frac{n}{\sqrt{5}}\left( \alpha^{n-2} - \beta^{n-2} \right).
\end{align*}

       
\subsection{Recursive formula}

We denote by $\mathbb{P}_G(v)$ the set of $P_3$ subtrees in $G$ that contain the vertex $v$. If there is no risk of confusion, we simply use $\mathbb{P}(v)$.
\begin{lem}
\label{Lem:Rec}
For any vertex $z$ in a graph $G$, we have
\begin{align}
\label{Eq:Rec}
Z_1(G)
=Z_1(G-z)+\sum_{v\in N_G(z)}Z_1(G-z-v)+\sum_{P\in \mathbb{P}_G(z)}Z_0(G-P).
\end{align}
\end{lem}
\begin{proof}
$Z_1(G-z)$ counts all the $1$-nearly independent edge subsets that do not contain $z$.
$\sum_{v\in N_G(z)}Z_1(G-z-v)$ counts all those that contain $z$ in a $P_2$.
$\sum_{P\in \mathbb{P}_G(z)}Z_0(G-P)$ counts those that contain $z$ in a $P_3$.
\end{proof}

A \emph{pseudo-leaf} of a forest $T$ is a vertex that is not isolated and has at most one neighbour that is not a leaf (a vertex of degree $1$). We often use \eqref{Eq:Rec} for $z$ being a leaf attached to a pseudo leaf $v$ with degree $d$ and neighbour $u$ of largest degree, so that it becomes
$$
Z_1(G)
=Z_1(G-z)+Z_1(G-z-v)+Z_0(G-N_G(v))+(d-2)Z_0(G-(N_G[v]\setminus \{u\}).
$$

\section{Trees of order $n$ with small $Z_1$}
\label{Sec:Min}
This section characterises the two trees of order $n\geq 9$ that have the smallest $Z_1$.
%
%
First, we show that the $n$-vertex tree with smallest $Z_1$ is only the star if $n\geq 9$. 
\begin{thm}
\label{Thm:Star_Min_tree}
Among all connected graphs, in particular trees, $T$ of order $n\geq 9$ we have 
$Z_1(T)\geq Z_1(K_{1,n-1})$, with equality if and only if $T$ is $K_{1,n-1}$.
\end{thm}

\begin{proof}
As discussed in Subsection \ref{Sub:REmE}, Removing an edge from a connected graph of order $n\geq 9$ decreases $Z_1$. Hence, we can restrict this proof to the case where $T$ is a tree.

The basis cases corresponding to $n=9$ and $10$ can be seen in Table \ref{table:trees-of-order-9} and Table \ref{table:trees-of-order-10}. Suppose that the claim holds for $n=k$, for some $k\geq 10$. Now consider the case of $n=k+1$. Let $v$ be a pseudo-leaf of degree $d$ in $T$, $z$ a leaf neighbour of $v$ and $u$ a neighbor of $v$ with largest degree that might possibly be not a leaf. If $d=n-1$, then $T \cong K_{1,n-1}$. So, we may assume that $d\le n-2$. Thus, we have
\begin{align*}
Z_1(T)=&Z_1(T-z)+Z_1(T-v-z)+Z_0(T-z-v-u)+(d-2)Z_0(T-N[v]\setminus \{u\})\\
=&Z_1(T-z)+Z_1(T-N[v]\setminus \{u\})+Z_0(T-N[v])+(d-2)Z_0(T-N[v]\setminus \{u\})\\
\geq& Z_1(K_{1,(n-1)-1})+Z_1(K_{1,(n-d-1)})+Z_0(K_{1,n-(d+1)-1})+(d-2)Z_0(K_{1,(n-d-1)})\\
=&\frac{(n-2)(n-3)}{2}+\frac{(n-d-1)(n-d-2)}{2} +n-(d+1)-1+1 \\&+ (d-2)(n-d-1+1)\\
=&-\frac{d^2}{2} + n^2 + \frac{5d}{2} - 5n + 3.
\end{align*}
Hence,
\begin{align*}
Z_1(T)-Z_1(K_{1,n-1})
&\ge -\frac{d^2}{2} + n^2 + \frac{5d}{2} - 5n + 3-\frac{(n-1)(n-2)}2\\
&=-\frac{d^2}2 + \frac{n^2}2 + \frac{5d}2 - \frac{7n}2 + 2\\
&=\frac{n^2}{2} - \frac{7n}{2}+2-\left( \frac{d^2}{2} -\frac{5d}{2} \right).
 \end{align*}
However, 
\begin{align*}
    \frac{d^2}{2} -\frac{5d}{2} \le \frac{(n-2)^2}{2} - \frac{5(n-2)}{2} = \frac{n^2}{2} -\frac{9n}{2} +7.
\end{align*}
Thus,
\begin{align*}
    Z_1(T)-Z_1(K_{1,n-1}) &\ge \frac{n^2}{2} - \frac{7n}{2}+2-\left( \frac{d^2}{2} -\frac{5d}{2} \right)\\
    &\geq  \frac{n^2}{2} - \frac{7n}{2}+2 - \left( \frac{n^2}{2} -\frac{9n}{2} +7  \right)=n-5> 0 &\text{ since } n> 5,
\end{align*}
completing the proof.
\end{proof}

We use similar techniques to find the tree with second-minimum $Z_1$. Let $B^k_n$ be the tree of order $n$ obtained from
a path, $P_k$, of order $k$ by adding $n - k$ new vertices and then joining them to exactly
one end-vertex of $P_k$. Such a tree is usually called a broom. $B^3_n$ is the only $n$-vertex tree with degree sequence $(n-2,2,1,\dots,1)$.

\begin{thm}
\label{Thm:2nd_Min_tree}
Among all trees $T\neq K_{1,n-1}$ of order $n\geq 9$ we have 
$Z_1(T)\geq Z_1(B^3_n)$, with equality if and only if $T$ is $B_n^3$.
\end{thm}
\begin{proof}
Note that 
$$
Z_1(B^3_n)
=Z_1(K_{1,n-2})+Z_1(K_{1,n-3})+1
=\frac{(n-2)(n-3)+(n-3)(n-4)+2}{2}
=(n-3)^2+1.
$$
The basis cases corresponding to $n=9$ and $10$ can be seen in Table \ref{table:trees-of-order-9} and Table \ref{table:trees-of-order-10}. Suppose that the claim holds for $n=k$, for some $k\geq 10$. Now consider the case of $n=k+1$. Let $v$ be a pseudo-leaf of degree $d$ in $T$, $z$ a leaf neighbour of $v$ and $u$ a neighbor of $v$ with largest degree that might possibly be not a leaf. If $T-z$ is a star, then $T$ is $B^3_n$, and there would be nothing left to prove. Hence, we can assume that $T-z$ is not a star. Since $T\neq K_{1,n-1}$, we must have $n \geq d+2$. For $n=d+2$ the graph $T$ is isomorphic to $B^3_n$. So, we can consider $n\geq d+3$.
Thus, we have
\begin{align*}
Z_1(T)=&Z_1(T-z)+Z_1(T-v-z)+Z_0(T-z-v-u)+(d-2)Z_0(T-N[v]\setminus \{u\})\\
=&Z_1(T-z)+Z_1(T-N[v]\setminus \{u\})+Z_0(T-N[v])+(d-2)Z_0(T-N[v]\setminus \{u\})\\
\geq& Z_1(B^3_{n-1})+Z_1(K_{1,(n-d-1)})+Z_0(K_{1,n-(d+1)-1})+(d-2)Z_0(K_{1,(n-d-1)})\\
=&(n-3-1)^2+1+\frac{(n-d-1)(n-d-2)}{2} +n-(d+1)-1+1 \\&+ (d-2)(n-d-1+1)\\
=&-\frac{d^2}{2} + \frac{3n^2}{2} + \frac{5d}{2} - \frac{21n}{2} + 17.
\end{align*}
Hence,
\begin{align*}
Z_1(T)-Z_1(B^3_n)
&\ge -\frac{d^2}{2} + \frac{3n^2}{2} + \frac{5d}{2} - \frac{21n}{2} + 17 - (n-3)^2-1\\
&= \frac{n^2}{2} - \frac{9n}{2} +7 - \left( \frac{d^2}{2} -\frac{5d}{2} \right).
\end{align*}
However, 
\begin{align*}
    \frac{d^2}{2} -\frac{5d}{2} \le \frac{(n-3)^2}{2} - \frac{5(n-3)}{2} = \frac{n^2}{2} -\frac{11n}{2} +12.
\end{align*}
Thus,
\begin{align*}
    Z_1(T)-Z_1(B^3_n) &\ge \frac{n^2}{2} - \frac{9n}{2} +7 - \left( \frac{d^2}{2} -\frac{5d}{2} \right)\\
    &\geq  \frac{n^2}{2} - \frac{9n}{2}+7 - \left( \frac{n^2}{2} -\frac{11n}{2} +12  \right)\\
    &=n-5> 0 &\text{ since } n> 5,
\end{align*}
completing the proof.
\end{proof}

\section{Forest with maximum $Z_1$}
\label{Sec:Max}

In this section we show that $P_n$ is the tree of order $n\geq 9$ that has the largest $Z_1$. We attempted to also determine the one that has the second-largest $Z_1$. We only managed to prove that it has to be a tripod (a tree with only three leaves). Based on a computer check for small values of $n$, a conjecture describing the full characterisation is provided.

We start with a few technical lemmas.
\begin{lem}
\label{Lem:Tech0}
Let $n$ be an integer.
\begin{itemize}
\item[i)] If $n\geq 0$, we have $Z_0(P_n)\geq n$.

\item[ii)] If $n\geq 4$, we have $Z_0(P_n)\geq n+1$.

\item[iii)] If $n\geq 3$, we have $Z_1(P_n)\geq n-2$.
\end{itemize}
\end{lem}

\begin{proof}
For $n=0,1,2$, we have $Z_0(P_0)=1\geq 0$, $Z_0(P_1)=1\geq 1$ and $Z_0(P_2)=2\geq 2$. Suppose that i) holds for $n=k\geq 2$, then for $n=k+1\geq 3$ we have
$$
Z_0(P_n)=Z_0(P_{k+1})=Z_0(P_k)+Z_0(P_{k-1})
\geq k+k-1 \geq k+1,
$$
since $k\geq 2$, thereby proving i).

For $n=4,5$, we have $Z_0(P_4)=5\geq 4+1$ and
$Z_0(P_5)=7\geq 5+1$. If $Z_0(P_n)\geq n+1$ for all $4\leq n \leq k$ for some $k\geq 5$, then
$$
Z_0(P_{k+1})=Z_0(P_k)+Z_0(P_{k-1})\geq k+1 +k-1+1\geq (k+1)+1.
$$
This proves ii).

The proof of iii), namely $Z_1(P_n)\geq n-2$ follows from the fact that $P_n$ has at least $n-2$ copies of $P_3$.
\end{proof}
\begin{lem}
\label{LEm:Tech1}
For any integers $n$ and $d$ with $n\geq d+1$ and $d\geq 5$, we have
$$
(d-1)Z_0(P_{n-d})+Z_1(P_{n-d})\leq Z_0(P_{n-3})+Z_1(P_{n-2}).
$$
\end{lem}
\begin{proof}
We proceed by induction on $n$. For $n=d+1\geq 4$, using Lemma \ref{Lem:Tech0} we have
\begin{align*}
(d-1)Z_0(P_1)+Z_1(P_1)&=d-1
\leq d+1-3+1\leq Z_0(P_{d+1-3})+Z_1(P_{d+1-2})\end{align*}
for $d\geq 5$. For $n=d+2\geq 5$, using Lemma \ref{Lem:Tech0} we have
\begin{align*}
(d-1)Z_0(P_2)+Z_1(P_2)&=2(d-1)
\leq d-1+1+d-2\leq Z_0(P_{d+2-3})+Z_1(P_{d+2-2})\end{align*}
for $d\geq 5$. 
Suppose that the claim holds for $n=k\geq d+2$. Now consider the case of $n=k+1$. Then we have
\begin{align*}
&Z_0(P_{n-3}) +Z_1(P_{n-2})\\
&= Z_0(P_{n-4})+Z_0(P_{n-5})+Z_1(P_{n-3})+Z_1(P_{n-4})+Z_0(P_{n-5})\\
&=Z_0(P_{(n-1)-3})+Z_0(P_{(n-2)-3})
+Z_1(P_{(n-1)-2})+Z_1(P_{(n-2)-2})+Z_0(P_{n-5})\\
&\geq (d-1)Z_0(P_{n-1-d})+Z_1(P_{n-1-d})
+(d-1)Z_0(P_{n-2-d})+Z_1(P_{n-2-d})+Z_0(P_{n-3-d})
\\
&=(d-1)Z_0(P_{n-d})+Z_0(P_{n-d}),
\end{align*}
as required.
\end{proof}
\begin{lem}
\label{Lem:Tech3}
For any integer $n\geq 7$, we have
$$
Z_1(P_{n-3})+Z_0(P_{n-4})+Z_0(P_{n-3})
\leq Z_1(P_{n-2})+Z_0(P_{n-3}).
$$
\end{lem}
\begin{proof}
We proceed by induction on $n \ge 7$. If $n=7$, we have
$$
Z_1(P_{4})+Z_0(P_{3})+Z_0(P_{4})=2+3+5
\leq 5+5= Z_1(P_{5})+Z_0(P_{4}).
$$
If $n=8$, we have
$$
Z_1(P_{5})+Z_0(P_{4})+Z_0(P_{5})=5+5+7
\leq 10+ 7= Z_1(P_{6})+Z_0(P_{5}).
$$
For the induction assumption, suppose that the inequality holds for $n=k\geq 8$. Suppose now that $n=k+1$. Then we have
\begin{align*}
&Z_1(P_{n-3})+Z_0(P_{n-4})+Z_0(P_{n-3})\\
&= Z_1(P_{n-4})+Z_1(P_{n-5})+Z_0(P_{n-6})+Z_0(P_{n-5})+Z_0(P_{n-6})+Z_0(P_{n-4})+Z_0(P_{n-5})\\
&= Z_1(P_{n-4})+Z_0(P_{n-5})+Z_0(P_{n-4})+
Z_1(P_{n-5})+Z_0(P_{n-6})+Z_0(P_{n-5})+Z_0(P_{n-6})\\
&= Z_1(P_{(n-1)-3})+Z_0(P_{(n-1)-4})+Z_0(P_{(n-1)-3})+\\
&+Z_1(P_{(n-2)-3})+Z_0(P_{(n-2)-4})+Z_0(P_{(n-2)-3})+Z_0(P_{n-6})\\
&\leq Z_1(P_{(n-1)-2})+Z_0(P_{(n-1)-3})
+ Z_1(P_{(n-2)-2})+Z_0(P_{(n-2)-3})+Z_0(P_{(n-2)-3})\\
&= Z_1(P_{n-2})+Z_0(P_{n-3}).
\end{align*}
as required.
\end{proof}
\begin{lem}
\label{Lem:Tech4}
For any integer $n\geq 7$, we have
$$
Z_1(P_{n-4})+Z_0(P_{n-5})+2Z_0(P_{n-4})
\leq Z_1(P_{n-2})+Z_0(P_{n-3}).
$$
\end{lem}
\begin{proof}
We proceed by induction on $n \ge 7$. For $n=7$, we have
$$
Z_1(P_{3})+Z_0(P_{2})+2Z_0(P_{3})
=1+2+2\times 3\leq 5+5 = Z_1(P_{5})+Z_0(P_{4}).
$$
For $n=8$, we have
$$
Z_1(P_{4})+Z_0(P_{3})+2Z_0(P_{4})
=2+3+2\times 5\leq 10+7 =Z_1(P_{6})+Z_0(P_{5}).
$$
For the induction assumption, suppose that the inequality holds for $n=k\geq 8$. Suppose now that $n=k+1$. Then we have
\begin{align*}
&Z_1(P_{n-4})+Z_0(P_{n-5})+2Z_0(P_{n-4})\\
&= Z_1(P_{n-5})+Z_1(P_{n-6})+Z_0(n-7)+Z_0(P_{n-6})+Z_0(P_{n-7})+2Z_0(P_{n-5})+2Z_0(P_{n-6})\\
&= Z_1(P_{n-5})+Z_0(P_{n-6})+2Z_0(P_{n-5})+Z_1(P_{n-6})+Z_0(P_{n-7})+2Z_0(P_{n-6}) +Z_0(P_{n-7})\\
&= Z_1(P_{(n-1)-4})+Z_0(P_{(n-1)-5})+2Z_0(P_{(n-1)-4})\\
&+Z_1(P_{(n-2)-4})+Z_0(P_{(n-2)-5})+2Z_0(P_{(n-2)-4})+Z_0(P_{(n-3)-4})\\
&\leq  Z_1(P_{(n-1)-2})+Z_0(P_{(n-1)-3})+Z_1(P_{(n-2)-2})+Z_0(P_{(n-2)-3})\\
&= Z_1(P_{n-2})+Z_0(P_{n-3}),
\end{align*}
as required.
\end{proof}

Since adding an edge can only increase $Z_1$ or keep it unchanged, for any forest $F$, there is a tree $T$ of the same order such that $Z_1(T)\geq Z_1(F)$.

The following lemma is well-known
\begin{lem}[\cite{Gutman1977AcyclicSW}]
\label{Lem:Z0MaxPat}
For any forest $F$ with order $n$ we have
$Z_0(F)\leq Z_0(P_n)$.
\end{lem}

We are now ready to present a proof of a characterisation of the forest of order $n\geq 9$ that has the largest $Z_1$.

\begin{thm}
\label{Thm:Path_Max_Forest}
Among all forests $F$ of order $n\geq 9$, we have 
$Z_1(F)\leq Z_1(P_n)$, with equality if and only if $F$ is $P_n$.
\end{thm}
\begin{proof}
We use an induction on $n$. The basis case of $n=9,10$ can be seen on the Tables \ref{table:trees-of-order-9} and \ref{table:trees-of-order-10}. Suppose that the claim holds for $n=k\geq 10$. We now consider the case of $n=k+1$. Suppose that $v$ is a pseudo-leaf of $F$ of degree $d$. Let $u$ be a neighbor of $v$ having the largest degree. If $v$ has a neighbour that is not a leaf, then it is $u$. Let $z$ be a leaf neighbour of $v$. 


{\bf Case 1:} Suppose that $d=2$. With the use of Lemma \ref{Lem:Z0MaxPat}, we have
\begin{align*}
Z_1(F)
&=Z_1(F-z)+Z_1(F-z-v) + Z_0(F-v-u-z)\\
&\leq Z_1(P_{n-1})+Z_1(P_{n-2})+Z_0(P_{n-3})
=Z_1(P_n).
\end{align*}

{\bf Case 2:} Suppose that $d=3.$ Let $x$ be the leaf adjacent to $v$ other than $z$. Then, using Lemma \ref{Lem:Tech3}, we have
\begin{align*}
Z_1(F)
&=Z_1(F-z)+Z_1(F-z-v) + Z_0(F-v-u-z)+ Z_0(F-v-x-z)\\
&=Z_1(F-z)+Z_1(F-z-v-x) + Z_0(F-v-u-z-x)+ Z_0(F-v-x-z)\\
&\leq Z_1(P_{n-1})+Z_1(P_{n-3})+Z_0(P_{n-4})+Z_0(P_{n-3})\\
&\leq Z_1(P_{n-1})+Z_1(P_{n-2})+Z_0(P_{n-3})=Z_1(P_n).
\end{align*}

{\bf Case 3:} Suppose that $d=4.$ Let $x$ and $y$ be the two leaves adjacent to $v$ other than $z$. Then, using Lemma \ref{Lem:Tech4}, we have
\begin{align*}
&Z_1(F)\\
&=Z_1(F-z)+Z_1(F-z-v) + Z_0(F-v-u-z)+ Z_0(F-v-x-z)+Z_0(F-v-y-z)\\
&=Z_1(F-z)+Z_1(F-z-v-x-y) + Z_0(F-v-u-z-x-y)+ 2Z_0(F-v-x-z-y)\\
&\leq Z_1(P_{n-1})+Z_1(P_{n-4})+Z_0(P_{n-5})+2Z_0(P_{n-4})\\
&\leq Z_1(P_{n-1})+Z_1(P_{n-2})+Z_0(P_{n-3})=Z_1(P_n).
\end{align*}

{\bf Case 4:} Suppose that $d\geq 5.$ By counting the $1$-nearly independent edge subsets without $v$, with $v$ in a $P_2$ and then with $v$ in a $P_3$, we have
\begin{align*}
Z_1(F)
&=Z_1(F-z)+ Z_1(F-(N[v]\setminus \{u\}))
+ Z_0(F-N[v])+(d-2)Z_0(F-(N(v)\setminus \{u\}))\\
&\leq Z_1(F-z)+ Z_1(F-(N[v]\setminus \{u\}))
+(d-1)Z_0(F-(N[v]\setminus \{u\}))\\
&\leq Z_1(P_{n-1})+ Z_1(P_{n-d})
+(d-1)Z_0(P_{n-d})\\
&\leq Z_1(P_{n-1})+ Z_1(P_{n-2})
+Z_0(P_{n-3})\qquad\text{(using Lemma \ref{LEm:Tech1})}\\
&=Z_1(P_n),
\end{align*}
completing the proof.
\end{proof}




From now, we aim to find out which $n$-vertex tree has the second-largest $Z_1$. A series of lemmas is needed.

We write $[T_1,\dots,T_j]$ for the rooted tree, where the branches of the root vertex $v$ are the rooted trees $T_1,\dots,T_j$, such that the root of each of $T_1,\dots,T_j$ is adjacent to $v$. The following lemma is well-known under the name of Ironing Lemma. It means that replacing a non-path branch by a path branch increases $Z_0$.

\begin{lem}[\cite{GUTMAN20122177}]
\label{Lem:Z0Ironing}
For any rooted  tree $T_1,\dots,T_j$, we have
$$
Z_0([T_1,\dots,T_j])<Z_0([P_{|T_1|},T_2,\dots,T_j])
$$
if $T_1$ is not a path rooted at one of its end-vertices.
\end{lem}

We now provide an ironing lemma for $Z_1$. Replacing a branch that is not a path by a path of the same order increases $Z_1$. 

\begin{lem}
\label{Lem:Z1Ironing}
For any rooted  trees $T_1,\dots,T_j$, we have
$$
Z_1([T_1,\dots,T_j])<Z_1([P_{|T_1|},T_2,\dots,T_j])
$$
if $T_1$ is not a path rooted at one of its end-vertices.
\end{lem}

\begin{proof}
We proceed by induction on $j$. If $j=1$, then $[P_{|T_1|}]=P_{|T_1|+1}$ and the claim holds by Theorem \ref{Thm:Path_Max_Forest}. Suppose that the claim holds whenever $j=k$ for some $k\geq 1$ and let us consider the case where $j=k+1$. Let $v$ be the root of $T=[T_1,\dots,T_j]$ and $v_i$ is the neighbour of $v$ in $T_i$ for any $i$. Suppose that $T_1$ is not a path rooted at one of its end-vertices. Then $|T_1|\geq 3$. In the equation below, we use Lemmas \ref{Lem:Z0Ironing} and \ref{Lem:Z1Ironing}, and the induction assumption to replace $T_1$ with $P_{|T_1|}$.
\begin{align*}
&Z_1(T)\\
&= Z_1(T-v_2) + Z_1 ((T_2-v_2)\cup (T-T_2-v))
+\sum_{x\in N_{T_2}(v_2)}Z_1(T-v_2-x)\\
&+\sum_{i\in \{1,3,4,\dots,\ell\}}Z_0(T-v_2-v-v_i)
+\sum_{x\in N_{T_2}(v_2)}Z_0(T-x-v_2-v)\\
&+\sum_{P\in \mathbb{P}_{T_2}(v_2)}Z_0(T-P)\\
&=Z_1([T_1,T_3,T_4,\dots,T_j])Z_0(T_2-v_2)
+ Z_0([T_1,T_3,T_4,\dots,T_j])Z_1(T_2-v_2)\\
&+Z_1(T_2-v_2)Z_0(T_1)\prod_{i=3}^jZ_0(T_i)
+Z_0(T_2-v_2)Z_1(T_1)\prod_{i=3}^jZ_0(T_i)
+Z_0(T_2-v_2)Z_0(T_1)Z_1(\bigcup_{i=3}^jT_i)\\
&+\sum_{x\in N_{T_2}(v_2)}Z_1(T_2-v_2-x)Z_0([T_1,T_3,T_4,\dots,T_j]) 
+Z_0(T_2-v_2-x)Z_1([T_1,T_3,T_4,\dots,T_j])\\
&+Z_0(T_1-v_1)Z_0(T-T_1-v_2-v)+\sum_{i\in \{3,4,\dots,\ell\}}Z_0(T_1)Z_0(T-T_1-v_2-v-v_i)\\
&+\sum_{x\in N_{T_2}(v_2)}Z_0(T_1)Z_0(T-T_1-x-v_2-v) +\sum_{P\in \mathbb{P}_{T_2}(v_2)}Z_0(T_2-P)Z_0(T_1)Z_0(T-T_1-T_2)\\
&<
Z_1([P_{|T_1|},T_3,T_4,\dots,T_j])Z_0(T_2-v_2)
+ Z_0([P_{|T_1|},T_3,T_4,\dots,T_j])Z_1(T_2-v_2)\\
&+Z_1(T_2-v_2)Z_0(P_{|T_1|})\prod_{i=3}^jZ_0(T_i)
+Z_0(T_2-v_2)Z_1(P_{|T_1|})\prod_{i=3}^jZ_0(T_i)\\
&+Z_0(T_2-v_2)Z_0(P_{|T_1|})Z_1(\bigcup_{i=3}^jT_i)
+\sum_{x\in N_{T_2}(v_2)}Z_1(T_2-v_2-x)Z_0([P_{|T_1|},T_3,T_4,\dots,T_j])\\ 
&+Z_0(T_2-v_2-x)Z_1([P_{|T_1|},T_3,T_4,\dots,T_j])
+Z_0(P_{|T_1-v_1|})Z_0(T-T_1-v_2-v)\\
&+\sum_{i\in \{3,4,\dots,\ell\}}Z_0(P_{|T_1|})Z_0(T-T_1-v_2-v-v_i)\\
&+\sum_{x\in N_{T_2}(v_2)}Z_0(P_{|T_1|})Z_0(T-T_1-x-v_2-v) +\sum_{P\in \mathbb{P}_{T_2}(v_2)}Z_0(T_2-P)Z_0(P_{|T_1|})Z_0(T-T_1-T_2)\\
&=Z_1([P_{|T_1|},T_2,\dots,T_j]),
\end{align*}
completing the proof.
\end{proof}

In view of Lemma \ref{Lem:Z1Ironing}, we can restrict to star-like trees when trying to find trees of given number of vertices and second-largest $Z_1$. A star-like tree is a tree with at most one vertex of degree greater than $2$. We write $[P_{n_1},\dots,P_{n_k}]$ for the star-like tree with $n_1+\dots+n_k+1$ vertices, $k$ branches where the $i$-th branches have length $n_i$ for all $i$.

In the following lemma, we replace two path branches with a single path rooted at one of its end.
\begin{lem}
\label{Lem:RedBrancNumb}
For any positive integers $n_1\geq \dots\geq n_j$ and $j\geq 3$, we have
$$
Z_1([P_{n_1},\dots,P_{n_j}])<Z_1([P_{n_1+n_2},P_{n_3},\dots,P_{n_j}]).
$$
\end{lem}
\begin{proof}
We reason by induction on the order $n=n_1+\dots+n_j+1$. If $n=j+1$, then $[P_{n_1},\dots,P_{n_j}]$ is a star and $[P_{n_1+n_2},P_{n_3},\dots,P_{n_j}]$ is not a star. The desired inequality holds by Theorem \ref{Thm:Star_Min_tree}.

Suppose that the inequality $Z_1([P_{n_1},\dots,P_{n_j}])<Z_1([P_{n_1+n_2},P_{n_3},\dots,P_{n_j}])$ holds whenever $n=n_1+\dots+n_j+1=k$ for some $k\geq j+1$. Now consider the case where $n=k+1\geq j+2$ Then $n_1\geq 2$. 

Suppose that $n_1>2$. Using Lemma \ref{Lem:Rec}, we have
\begin{align*}
Z_1([P_{n_1},\dots,P_{n_j}])
&=Z_1([P_{n_1-1},\dots,P_{n_j}])
+Z_1([P_{n_1-2},\dots,P_{n_j}])
+Z_0([P_{n_1-3},\dots,P_{n_j}])
\end{align*}
and
\begin{align}
\label{Eq:Z1PComb}
Z_1([P_{n_1+n_2},P_{n_3},\dots,P_{n_j}])
&=Z_1([P_{n_1+n_2-1},P_{n_3},\dots,P_{n_j}])
+Z_1([P_{n_1+n_2-2},P_{n_3},\dots,P_{n_j}])\nonumber\\
&+Z_0([P_{n_1+n_2-3},P_{n_3},\dots,P_{n_j}]).
\end{align}

By Lemma \ref{Lem:Z0Ironing}, we know that 
$$
Z_0([P_{n_1+n_2-3},P_{n_3},\dots,P_{n_j}])>Z_0([P_{n_1-3},P_{n_2-3},\dots,P_{n_j}]).
$$
By the induction assumption, we know that 
$$
Z_1([P_{n_1-1},\dots,P_{n_j}])< Z_1([P_{n_1+n_2-1},P_{n_3},\dots,P_{n_j}])
$$
and
$$
Z_1([P_{n_1-2},\dots,P_{n_j}])< Z_1([P_{n_1+n_2-2},P_{n_3},\dots,P_{n_j}]).
$$
Hence, we have
$$
Z_1([P_{n_1},\dots,P_{n_j}])< Z_1([P_{n_1+n_2},P_{n_3},\dots,P_{n_j}])
$$
as we aimed to prove.

Now suppose that $n_1=2$. \eqref{Eq:Z1PComb} still holds, while
\begin{align*}
Z_1([P_{n_1},\dots,P_{n_j}])
&=Z_1([P_{n_1-1},\dots,P_{n_j}])
+Z_1([P_{n_1-2},\dots,P_{n_j}])
+Z_0\left(\bigcup_{i=2}^jP_{n_i}\right).
\end{align*}
Note that $n_1-2=0,n_1+n_2-2=n_2$ and hence 
$$[P_{n_1-2},\dots,P_{n_j}]=[P_{n_2},\dots,P_{n_j}]=[P_{n_1+n_2-2},\dots,P_{n_j}].$$ Moreover, $\bigcup_{i=2}^jP_{n_i}$ can be obtained from $[P_{n_1+n_2-3},P_{n_3},\dots,P_{n_j}]=[P_{n_2-1},P_{n_3},\dots,P_{n_j}]$ by removing all edges incident to the branching vertex except the one connecting it to $P_{n_2}$. Thus we have
$$
Z_0([P_{n_1+n_2-3},P_{n_3},\dots,P_{n_j}])>Z_0\left(\bigcup_{i=2}^jP_{n_i}\right).
$$
By the induction assumption, we know that 
$$
Z_1([P_{n_1-1},\dots,P_{n_j}])< Z_1([P_{n_1+n_2-1},P_{n_3},\dots,P_{n_j}])
$$
and
$$
Z_1([P_{n_1-2},\dots,P_{n_j}])= Z_1([P_{n_1+n_2-2},P_{n_3},\dots,P_{n_j}]).
$$
Hence, we again have
$$
Z_1([P_{n_1},\dots,P_{n_j}])< Z_1([P_{n_1+n_2},P_{n_3},\dots,P_{n_j}])
$$
as desired.
\end{proof}

The following theorem follows immediately from Lemmas \ref{Lem:Z1Ironing} and \ref{Lem:RedBrancNumb}. It reduces the set of candidates to that of tripods.

\begin{thm}
If a tree $T$ has order $n\geq 4$, and for any $n$-vertex tree $H$, we have $Z_1(P_n)>Z_1(T)\geq Z_1(H)$, then
$$
T\in \{[P_{n_1},P_{n_2},P_{n_3}]:n_1+n_2+n_3=n-1\}.
$$
\end{thm}
Tables \ref{table:trees-of-order-9} and \ref{table:trees-of-order-10} suggest that the $n$-vertex trees with second-largest $Z_1$ is $[P_1,P_1,P_{n-3}]$ for $n=9$, and it is $[P_3,P_3,P_{n-7}]$ for $n=10$. Further computational check showed that we have $[P_1,P_1,P_{n-3}]$ again for $n=11$, but for $12\leq n \leq 20$ we always have $[P_3,P_3,P_{n-7}]$. Hence, the following conjecture.
\begin{conj}
\label{Conj:2nd_Max_Forest}
Among all forest $F\neq P_n$ of order $n\geq 12 $ we have 
$Z_1(F)\leq Z_1([P_3,P_3,P_{n-7}])$, with equality if and only if $F$ is $[P_3,P_3,P_{n-7}]$.
\end{conj}

The fact that $[P_3,P_3,P_{n-7}]$ is not the tripod with largest $Z_0$ \cite{TEric} is part of the reason why proving Conjecture \ref{Conj:2nd_Max_Forest} is challenging. 

\newpage

\section{Appendix}

In this appendix, we exhaustively compute $Z_1(T)$, where $T$ is any tree of order $n$, where
$9 \le n \le 10$. 
The trees of order $n=9$ are in Table \ref{table:trees-of-order-9}.  Those of order $n=10$, are in in Table \ref{table:trees-of-order-10}.

 \begin{table}[!h]
   \centering
\begin{tabular}{|l|c|l|c|l|c|}
       \hline
        \textbf{$\mathbf{T}$} & $\mathbf{Z_1(T)}$ & \textbf{ $\mathbf{T}$} & $\mathbf{Z_1(T)}$ & \textbf{$\mathbf{T}$} & $\mathbf{Z_1(T)}$ \\
\hline
  
   \end{tabular}
 \end{table}

 \newpage

 \begin{table}[!h]
   \centering
   \begin{tabular}{|l|c|l|c|l|c|}
       \hline
        \textbf{ $\mathbf{T}$} & $\mathbf{Z_1(T)}$ & \textbf{ $\mathbf{T}$} & $\mathbf{Z_1(T)}$ & \textbf{ $\mathbf{T}$} & $\mathbf{Z_1(T)}$ \\
\hline
  
   \end{tabular}
   \caption{$Z_1$ of trees of order $n=9$.}
   \label{table:trees-of-order-9}

 \end{table}

 \newpage


 \begin{table}[!h]
   \centering
   \begin{tabular}{|l|c|l|c|l|c|}
       \hline
        \textbf{ $\mathbf{T}$} & $\mathbf{Z_1(T)}$ & \textbf{$\mathbf{T}$} & $\mathbf{Z_1(T)}$ & \textbf{$\mathbf{T}$} & $\mathbf{Z_1(T)}$ \\
\hline
   \end{tabular}
 \end{table}

 \newpage


 \begin{table}[!h]
   \centering
   \begin{tabular}{|l|c|l|c|l|c|}
       \hline
        \textbf{ $\mathbf{T}$} & $\mathbf{Z_1(T)}$ & \textbf{$\mathbf{T}$} & $\mathbf{Z_1(T)}$ & \textbf{ $\mathbf{T}$} & $\mathbf{Z_1(T)}$ \\
\hline
  
   \end{tabular}
 \end{table}

 \newpage

 \begin{table}[!h]
   \centering
   \begin{tabular}{|l|c|l|c|l|c|}
       \hline
        \textbf{ $\mathbf{T}$} & $\mathbf{Z_1(T)}$ & \textbf{ $\mathbf{T}$} & $\mathbf{Z_1(T)}$ & \textbf{ $\mathbf{T}$} & $\mathbf{Z_1(T)}$ \\
\hline

   \end{tabular}
 \end{table}

 \newpage

 \begin{table}[!h]
   \centering
   \begin{tabular}{|l|c|l|c|l|c|}
       \hline
        \textbf{ $\mathbf{T}$} & $\mathbf{Z_1(T)}$ & \textbf{ $\mathbf{T}$} & $\mathbf{Z_1(T)}$ & \textbf{ $\mathbf{T}$} & $\mathbf{Z_1(T)}$ \\
\hline

   \end{tabular}
 \end{table}

 \newpage


 \begin{table}[!h]
   \centering
   \begin{tabular}{|l|c|l|c|l|c|}
       \hline
        \textbf{ $\mathbf{T}$} & $\mathbf{Z_1(T)}$ & \textbf{ $\mathbf{T}$} & $\mathbf{Z_1(T)}$ & \textbf{ $\mathbf{T}$} & $\mathbf{Z_1(T)}$ \\
\hline
  
   \end{tabular}
   \caption{$Z_1$ of trees of order $n=10$.}
   \label{table:trees-of-order-10}
 \end{table}


\bibliographystyle{abbrv} 
\bibliography{references}

\end{document}